\documentclass{amsart}
\usepackage[english]{babel}
\usepackage{amssymb,enumerate,bbm,amsmath}
\usepackage[colorlinks=true,linkcolor=blue,citecolor=blue,urlcolor=blue]{hyper ref}

\newtheorem{theo}{Theorem}
\newtheorem{pro}[theo]{Proposition}
\newtheorem{lem}[theo]{Lemma}

\newtheorem*{thmA}{Theorem A}
\newtheorem*{thmB}{Theorem B}
\newtheorem*{thmC}{Theorem C}

\renewcommand{\(}{\left(}
\renewcommand{\)}{\right)}
\renewcommand{\~}{\tilde}
\renewcommand{\-}{\overline}

\newcommand{\R}{\mathbb{R}}

\renewcommand{\S}{\mathbb{S}}
\renewcommand{\H}{\mathbb{H}}

\renewcommand{\d}{\delta}

\renewcommand{\k}{\kappa}
\renewcommand{\l}{\lambda}

\newcommand{\D}{\Delta}

\renewcommand{\t}{\theta}
\newcommand{\s}{\sigma}
\newcommand{\p}{\varphi}

\newcommand{\ra}{\rightarrow}

\begin{document}

\title[Willmore inequality on hypersurfaces in hyperbolic space]{Willmore inequality on hypersurfaces in hyperbolic space}
\author{Yingxiang Hu}
\address{Center of Mathematical Sciences \\ Zhejiang University \\ Hangzhou 310027 \\ China\\}
\email{huyingxiang10@163.com}
\date{}
\thanks{}
\begin{abstract}
In this article, we prove a geometric inequality for star-shaped and mean-convex hypersurfaces in hyperbolic space by inverse mean curvature flow. This inequality can be considered as a generalization of Willmore inequality for closed surface in hyperbolic $3$-space.
\end{abstract}

{\maketitle}
\section{Introduction}
The classical isoperimetric inequality and its generalization, the Alexandrov-Fenchel inequalities play an important role in different branches of geometry. Let $\Omega \subset \R^n$ be a smooth bounded domain with boundary $\Sigma$, then the classical isoperimetric inequality is
\begin{align}\label{1}
|\Sigma| \geq n^\frac{n-1}{n}\omega_{n-1}^{\frac{1}{n}}|\Omega|^\frac{n-1}{n},
\end{align}
and equality in (\ref{1}) holds if and only if $\Omega$ is a geodesic ball.

For $k\in\{1,\cdots,n-1\}$, we denote by $p_k$ the normalized $k$-th order mean curvature of $\Sigma$, and set $p_0=1$ by convention. The celebrated Alexandrov-Fenchel inequalities \cite{Alexandrov1937,Alexandrov1938,Fenchel1936} for convex hypersurface $\Sigma^{n-1} \subset \R^n$ are
\begin{align}\label{2}
\frac{1}{\omega_{n-1}}\int_{\Sigma}p_k d\mu \geq \(\omega_{n-1}\int_{\Sigma}p_j d\mu \)^{\frac{n-1-k}{n-1-j}}, \quad 0\leq j<k \leq n-1,
\end{align}
and equality in (\ref{2}) holds if and only if $\Omega$ is a geodesic ball.

Observe that the isoperimetric inequality holds for non-convex domains, it is natural to extend the original Alexandrov-Fenchel inequality to non-convex domains, see \cite{Gibbons1997,Trudinger1994,Guan-Li2009,Chang-Wang2011,Chang-Wang2013,Chang-Wang2014,Qiu2015}. We should also mention that the Willmore inequality, which is a weaker form of Alexandrov-Fenchel inequality, has been established for closed surfaces in $\R^3$. More precisely, for any closed surface $\Sigma \subset \R^3$, the Willmore inequality \cite{Chen1971,Li-Yau1982,Topping1998} is
\begin{align}\label{Willmore-Euclidean-2dim}
\int_{\Sigma}p_1^2d\mu \geq \omega_2,
\end{align}
and equality in (\ref{Willmore-Euclidean-2dim}) holds if and only if $\Sigma$ is a geodesic sphere.

It is interesting to establish the Alexandrov-Fenchel inequalities for hypersurfaces in hyperbolic space, see \cite{Borisenko-Miquel1999,Gallego-Solanes2005}. Recently, the following hyperbolic Alexandrov-Fenchel inequalities were obtained.
\begin{thmA}[\cite{Ge-Wang-Wu2013,Ge-Wang-Wu2014,Wang-Xia2014}]
Let $k\in \{1,\cdots,n-1\}$. Any horospherical convex hypersurface $\Sigma\subset\H^n$ satisfies
\begin{align}\label{3}
\int_{\Sigma}p_{k}d\mu\geq \omega_{n-1}\left[\(\frac{|\Sigma|}{\omega_{n-1}}\)^\frac{2}{k}+\(\frac{|\Sigma|}{\omega_{n-1}}\)^\frac{2(n-1-k)}{k(n-1)} \right]^\frac{k}{2}.
\end{align}
Equality in (\ref{3}) holds if and only if $\Sigma$ is a geodesic sphere.
\end{thmA}
Inequality (\ref{3}) was proved in \cite{Ge-Wang-Wu2013} for $k=4$ and in \cite{Ge-Wang-Wu2014} for general even $k$. For $k=1$, (\ref{3}) was proved in \cite{Ge-Wang-Wu2014} with a help of a result of Cheng and Zhou \cite{Cheng-Zhou2012}. For general integer $k$, (\ref{3}) was proved in \cite{Wang-Xia2014}.

For $k=2$, inequality (\ref{3}) was proved by Li-Wei-Xiong \cite{Li-Wei-Xiong2014} under a weaker condition that $\Sigma$ is star-shaped and two-convex. More precisely,
\begin{thmB}[\cite{Li-Wei-Xiong2014}]
Any star-shaped and $2$-convex hypersurface $\Sigma\subset \H^n$$(n\geq 3)$ satisfies
\begin{align}\label{4}
\int_{\Sigma}p_2 d\mu \geq \omega_{n-1}^\frac{2}{n-1}|\Sigma|^\frac{n-3}{n-1}+|\Sigma|.
\end{align}
Equality in (\ref{4}) holds if and only if $\Sigma$ is a geodesic sphere.
\end{thmB}

The Willmore inequality (\ref{Willmore-Euclidean-2dim}) has also been generalized to closed surface $\Sigma\subset \H^3$.
\begin{thmC}[\cite{Chen1974,Maeda1978,Ritore2005}] Any closed surface $\Sigma \subset \H^3$ satisfies
\begin{align}\label{willmore-3dim}
\int_{\Sigma} p_1^2 d\mu \geq \omega_2+|\Sigma|.
\end{align}
Equality in (\ref{willmore-3dim}) holds if and only if $\Sigma$ is a geodesic sphere.
\end{thmC}

Inspired by these previous results, we prove the following geometric inequality, which can be considered as a generalization of Willmore inequality for hypersurfaces in hyperbolic space.
\begin{theo}\label{main-theorem}
Let $\Sigma\subset \H^n$ $(n\geq 3)$ be star-shaped and mean-convex hypersurface, then
\begin{align}\label{5}
\int_{\Sigma}p_1^2 d\mu \geq \omega_{n-1}^\frac{2}{n-1}|\Sigma|^{\frac{n-3}{n-1}}+|\Sigma|.
\end{align}
Equality in (\ref{5}) holds if and only if $\Sigma$ is a geodesic sphere.
\end{theo}

We expect that the inequality (\ref{5}) will be useful in defining the Hawking mass for hypersurfaces in $\H^n$. In \cite{Hung-Wang2015}, the {\em Hawking mass} for a closed embedded surface $\Sigma$ in $\H^3$ is defined as
\begin{align*}
m_H(\Sigma)=\frac{|\Sigma|^\frac{1}{2}}{2\omega_2^\frac{1}{2}}\left[1-\omega_2^{-1}\int_{\Sigma}(p_1^2-1)d\mu\right].
\end{align*}

We now give the outline of the proof of Theorem \ref{main-theorem}. Motivated by \cite{Brendle-Hung-Wang2016,Lima-Girao2015,Li-Wei-Xiong2014}, we adopt the inverse mean curvature flow (IMCF) in our proof. This flow has been used by Huisken and Ilmanen \cite{Huisken-Ilmanen1997,Huisken-Ilmanen2001} to prove the Riemannian Penrose inequality in general relativity. We start from a given star-shaped and mean-convex hypersurface $\Sigma$, and evolve it by IMCF. By the convergence results of Gerhardt \cite{Gerhardt2011} (see also \cite{Brendle-Hung-Wang2016}), the IMCF exists for all time, and the evolving hypersurface $\Sigma_t$ with $\Sigma_0=\Sigma$ remains star-shaped and mean-convex for all $t\geq 0$.

We next consider the quantity:
\begin{align*}
Q(t):=|\Sigma_t|^{-\frac{n-3}{n-1}}\int_{\Sigma_t}(p_1^2-1) d\mu.
\end{align*}
We study the limit of $Q(t)$ as $t\ra \infty$. Notice that the roundness estimate for $\Sigma_t$ is not strong enough to calculate the limit of $Q(t)$. However, similar to \cite{Brendle-Hung-Wang2016,Li-Wei-Xiong2014}, we are able to give a positive lower bound for the limit of $Q(t)$, which will be used to establish the monotonicity of $Q(t)$. Finally, we prove that $Q(t)$ is monotone non-increasing under IMCF. From this, Theorem \ref{main-theorem} follows immediately.

\section{Preliminaries}
In this article, we consider the hyperbolic space $\H^n=\R^{+}\times \S^{n-1}$ equipped with the metric
$$
\-g=dr^2+\sinh^2 r g_{\S^{n-1}},
$$
where $g_{\S^{n-1}}$ is the standard round metric on the unit sphere $\S^{n-1}$. Let $\Sigma\subset \H^n$ be a closed hypersurface with its unit outward normal vector $\nu$. The second fundamental form $h$ of $\Sigma$ is defined by
\begin{align*}
h(X,Y)=\langle \-\nabla_X \nu,Y \rangle
\end{align*}
for any $X,Y\in T\Sigma$. The principal curvature $\k=(\k_1,\cdots,\k_n)$ are the eigenvalues of $h$ with respect to the induced metric $g$ on $\Sigma$. For $k\in \{1,\cdots,n-1\}$, the normalized $k$-th elementary symmetric polynomial of $\k$ is defined as
\begin{align*}
p_k(\k):=\frac{1}{\binom{n-1}{k}}\sum_{i_1<i_2<\cdots <i_k}\k_{i_1}\cdots \k_{i_k},
\end{align*}
which can also be viewed as a function of the second fundamental form $h_{i}^{j}=g^{jk}h_{ki}$. For abbreviation, we write $p_k$ for $p_k(\k)$.

We now consider the inverse mean curvature flow (IMCF)
\begin{align}\label{6}
\partial_t X=\frac{1}{(n-1)p_1}\nu.
\end{align}
where $\Sigma_t=X(t,\cdot)$ is a family of hypersurfaces in $\H^n$, $\nu$ is the unit outward normal to $\Sigma_t=X(t,\cdot)$. Let $d \mu_t$ be its area element on $\Sigma_t$. We list the following evolution equations.
\begin{lem}Under IMCF (\ref{6}), we have:
\begin{align}\label{7}
\partial_t p_1 =&-\frac{1}{(n-1)^2}\D\(\frac{1}{p_1}\)-\frac{1}{(n-1)^2p_1}(|A|^2+n-1).
\end{align}
\begin{align}\label{8}
\partial_t d\mu =&d\mu.
\end{align}
\end{lem}

In \cite{Gerhardt2011}, Gerhardt investigated the inverse curvature flow of star-shaped hypersurfaces in hyperbolic space.
\begin{theo}[\cite{Gerhardt2011}]\label{theo-Gerhardt}
If the initial hypersurface is star-shaped and mean-convex, then the solution for IMCF (\ref{6}) exists for all time $t$ and preserves the condition of star-shapedness and mean-convexity. Moreover, the hypersurfaces become strictly convex exponentially fast and more and more totally umbilical in the sense of
\begin{align*}
|h_i^j-\d_i^j| \leq C e^{-\frac{t}{n-1}},\quad t>0,
\end{align*}
i.e., the principal curvatures are uniformly bounded and converge exponentially fast to one.
\end{theo}

\section{The asymptotic behavior of monotone quantity}
We define the quantity
\begin{align*}
Q(t):=|\Sigma_t|^{-\frac{n-3}{n-1}}\int_{\Sigma_t}(p_1^2-1) d\mu,
\end{align*}
where $|\Sigma_t|$ is the area of $\Sigma_t$. In this section, we estimate the lower bound of the limit of $Q(t)$.  First of all, we recall the following sharp Sobolev inequality on $\S^{n-1}$ due to Beckner \cite{Beckner1993}. This Sobolev inequality is crucial in analyzing the asymptotic behavior of the monotone quantity, see \cite{Brendle-Hung-Wang2016,Li-Wei-Xiong2014}, etc.
\begin{lem}\label{Beckner}For every positive function $f$ on $\S^{n-1}$, we have
\begin{equation}\label{Beckner-inequality-i}
\begin{split}
     &\int_{\S^{n-1}}f^{n-3}dvol_{\S^{n-1}}+\frac{n-3}{n-1}\int_{\S^{n-1}}f^{n-5}|\nabla f|^2 dvol_{\S^{n-1}} \\
\geq &\omega_{n-1}^{\frac{2}{n-1}}\(\int_{\S^{n-1}} f^{n-1}dvol_{\S^{n-1}}\)^\frac{n-3}{n-1}.
\end{split}
\end{equation}
Moreover, equality in (\ref{Beckner-inequality-i}) holds if and only if $f$ is a constant.
\end{lem}
\begin{proof}
From Theorem 4 in \cite{Beckner1993}, for any positive smooth function $w$ on $\S^{n-1}$, we have the following inequality
\begin{equation}\label{Beckner-inequality-ii}
\begin{split}
     &\frac{4}{(n-1)(n-3)}\int_{\S^{n-1}}|\nabla w|^2 dvol_{\S^{n-1}}+\int_{\S^{n-1}} w^2 dvol_{\S^{n-1}} \\
\geq &\omega_{n-1}^{\frac{2}{n-1}}\(\int_{\S^{n-1}} w^{\frac{2(n-1)}{n-3}}dvol_{\S^{n-1}}\)^\frac{n-3}{n-1}.
\end{split}
\end{equation}
Moreover, the equality holds if and only if $w$ is a constant. For any positive function $f$ on $\S^{n-1}$, by letting $w=f^{\frac{n-3}{2}}$, we obtain the desired estimate.
\end{proof}

\begin{pro}\label{lower-bound-estimate}
Under IMCF (\ref{6}), we have
\begin{align}\label{19}
\liminf_{t\ra \infty} Q(t)\geq \omega_{n-1}^\frac{2}{n-1}.
\end{align}
\end{pro}
\begin{proof}
Recall that star-shaped hypersurfaces can be written as graphs of function $r=r(t,\t)$, $\t\in \S^{n-1}$. Denote $\l(r)=\sinh(r)$, then $\l'(r)=\cosh(r)$. We next define a function $\p(\t)=\Phi(r(\t))$, where $\Phi(r)$ is a positive function satisfying $\Phi'=\frac{1}{\l}$. Let $\t=\left\{\t^j\right\}$, $j=1,\cdots,n-1$ be a coordinate system on $\S^{n-1}$ and $\p_i,\p_{ij}$ be the covariant derivatives of $\p$ with respect to the metric $g_{\S^{n-1}}$. Define
$$
v=\sqrt{1+|\nabla \p|_{\S^{n-1}}^2}.
$$
From \cite{Gerhardt2011}, we know that
\begin{align}\label{20}
\l=O(e^\frac{t}{n-1}),\quad |\nabla \p|_{\S^{n-1}}+|\nabla^2 \p|_{\S^{n-1}}=O(e^{-\frac{t}{n-1}}).
\end{align}
Since $\l'=\sqrt{1+\l^2}$, we have
\begin{align}\label{21}
\l'=\l\(1+\frac{1}{2}\l^{-2}+O(e^{-\frac{4t}{n-1}})\).
\end{align}
From (\ref{20}), we also have
\begin{align}\label{22}
\frac{1}{v}=1-\frac{1}{2}|\nabla \p|_{\S^{n-1}}^2+O(e^{-\frac{4t}{n-1}}).
\end{align}
In terms of $\p$, we can express the metric and the second fundamental form of $\Sigma$ as follows,
\begin{align*}
g_{ij}=&\l^2(\s_{ij}+\p_i\p_j),\\
h_{ij}=&\frac{\l'}{v\l}g_{ij}-\frac{\l}{v}\p_{ij},
\end{align*}
where $\s_{ij}=g_{\S^{n-1}}(\partial_{\t^i},\partial_{\t^j})$. Denote $a_i=\sum_k \s^{ik}\p_{ki}$ and note that $\sum_i a_i=\D_{\S^{n-1}}\p$. By (\ref{20}), the principal curvatures of $\Sigma_t$ has the following form
$$
\k_i=\frac{\l'}{v\l}-\frac{a_i}{v\l}+O(e^{-\frac{4t}{n-1}}),\quad i=1,\cdots,n-1.
$$
Then we have
$$
p_1=\frac{\l'}{v\l}-\frac{\D_{\S^{n-1}}\p}{(n-1)v\l}+O(e^{-\frac{4t}{n-1}}).
$$
By using (\ref{21}) and (\ref{22}), we get
$$
p_1=1+\frac{1}{2\l^2}-\frac{|\nabla\p|_{\S^{n-1}}^{2}}{2}-\frac{\D_{\S^{n-1}}\p}{(n-1)\l}+O(e^{-\frac{4t}{n-1}}).
$$
and
\begin{align*}
p_1^2-1=\frac{1}{\l^2}-|\nabla\p|_{\S^{n-1}}^{2}-\frac{2}{n-1}\frac{\D_{\S^{n-1}}\p}{\l}+O(e^{-\frac{4t}{n-1}}).
\end{align*}
On the other hand,
$$
\sqrt{\det g}=\left[\l^{n-1}+O(e^\frac{(n-3)t}{n-1})\right]\sqrt{\det g_{\S^{n-1}}}.
$$
So we have
\begin{align*}
\int_{\Sigma_t}(p_1^2-1)d\mu=&\int_{\S^{n-1}}\l^{n-1}(p_1^2-1) dvol_{\S^{n-1}}+O(e^\frac{(n-5)t}{n-1}) \\
=&\int_{\S^{n-1}}\(\l^{n-3}-\l^{n-1}|\nabla \p|_{\S^{n-1}}^2\)dvol_{\S^{n-1}} \\
 &-\frac{2}{n-1}\int_{\Sigma}\l^{n-2}\D_{\S^{n-1}}\p dvol_{\S^{n-1}}+O(e^{\frac{(n-5)t}{n-1}}) \\
=&\int_{\S^{n-1}}\(\l^{n-3}-\l^{n-1}|\nabla \p|_{\S^{n-1}}^2\)dvol_{\S^{n-1}} \\
 &+\frac{2(n-2)}{n-1}\int_{\Sigma}\l^{n-3}\langle\nabla \l, \nabla \p\rangle_{\S^{n-1}} dvol_{\S^{n-1}}+O(e^{\frac{(n-5)t}{n-1}}).
\end{align*}
Since $\nabla \l=\l \l'\nabla \p$, it follows that $|\nabla \l-\l^2\nabla \p|_{g_{\S^{n-1}}}\leq O(e^{-\frac{t}{n-1}})$, we deduce that
\begin{equation}\label{23}
\int_{\Sigma_t}(p_1^2-1)d\mu=\int_{\S^{n-1}}\(\l^{n-3}+\frac{n-3}{n-1}\l^{n-5}|\nabla\l|^2\)dvol_{\S^{n-1}}+O(e^\frac{(n-5)t}{n-1}).
\end{equation}
Moreover,
$$
|\Sigma_t|^{\frac{n-3}{n-1}}=\(\int_{\S^{n-1}}\l^{n-1}dvol_{\S^{n-1}}\)^\frac{n-3}{n-1}+O(e^\frac{(n-5)t}{n-1}).
$$
Using Lemma \ref{Beckner}, we achieve
$$
\liminf_{t\ra \infty}|\Sigma_t|^{-\frac{n-3}{n-1}}\int_{\Sigma_t}(p_1^2-1)d\mu\geq \omega_{n-1}^\frac{2}{n-1}.
$$
\end{proof}

\section{Monotonicity}
In this section, we show that $Q(t)$ is monotone non-increasing under IMCF (\ref{6}).
\begin{pro}\label{pro-monotonicity}
Under IMCF (\ref{6}), the quantity $Q(t)$ is monotone non-increasing. Moreover, $\frac{d}{dt}Q(t)=0$ at some time $t$ if and only if $\Sigma_t$ is totally umbilical.
\end{pro}
\begin{proof}
Under IMCF (\ref{6}), we have
\begin{align*}
 &\frac{d}{dt}\int_{\Sigma_t}(p_1^2-1) d\mu \\
=&\int_{\Sigma_t}2p_1\partial_t p_1d\mu +\int_{\Sigma_t}(p_1^2-1)d\mu \\
=&-\frac{2}{(n-1)^2}\int_{\Sigma_t}p_1\left[\D\(\frac{1}{p_1}\)+\frac{1}{p_1}(|A|^2-(n-1))\right] d\mu+\int_{\Sigma_t}(p_1^2-1)d\mu\\
=&-\frac{2}{(n-1)^2}\int_{\Sigma_t}\left[\frac{1}{p_1^2}|\nabla p_1|^2+(|A|^2-(n-1))\right]d\mu+\int_{\Sigma_t}(p_1^2-1)d\mu\\
\leq &\frac{n-3}{n-1}\int_{\Sigma_t}(p_1^2-1)d\mu,
\end{align*}
where the last inequality follows from the trace inequality $|A|^2 \geq (n-1)p_1^2$. Combining with Proposition \ref{lower-bound-estimate}, we know that the quantity
$$
\int_{\Sigma_t}(p_1^2-1)d\mu
$$
is positive under IMCF (\ref{6}). Therefore, combining with (\ref{8}) we get
$$
\frac{d}{dt}Q(t)\leq 0.
$$
If the equality holds, then the trace inequality implies that $\Sigma_t$ is totally umbilical.
\end{proof}

Now we finish the proof of Theorem \ref{main-theorem}.
\begin{proof}[Proof of Theorem \ref{main-theorem}]
Since $Q(t)$ is monotone non-increasing, we have
\begin{align}
Q(0)\geq \liminf_{t\ra \infty} Q(t)\geq \omega_{n-1}^\frac{2}{n-1}.
\end{align}
This implies that $\Sigma_0=\Sigma$ satisfies
\begin{align*}
\int_{\Sigma}(p_1^2-1)d\mu \geq \omega_{n-1}^\frac{2}{n-1}|\Sigma|^\frac{n-3}{n-1},
\end{align*}
which is equivalent to
\begin{align*}
\int_{\Sigma}p_1^2 d\mu \geq \omega_{n-1}^\frac{2}{n-1}|\Sigma|^\frac{n-3}{n-1}+|\Sigma|.
\end{align*}
Now if we assume that equality in (\ref{5}) is attained, then $Q(t)$ is a constant. Then Proposition \ref{pro-monotonicity} indicates that $\Sigma_t$ is totally umbilical and therefore a geodesic sphere. If $\Sigma$ is a geodesic sphere of radius $r$, then $|\Sigma|=\omega_{n-1}\sinh^{n-1}r$ and $p_1=\coth r$.
Hence, we have
\begin{align*}
\int_{\Sigma}p_1^2d\mu = \omega_{n-1}\sinh^{n-1}r \coth^2 r=\omega_{n-1}^\frac{2}{n-1}|\Sigma|^\frac{n-3}{n-1}+|\Sigma|.
\end{align*}
Therefore, equality in (\ref{5}) holds on a geodesic sphere. This finishes the proof of Theorem \ref{main-theorem}.
\end{proof}

\end{document}